\documentclass[dvips]{imsart}

\RequirePackage[OT1]{fontenc}
\usepackage{bbm}
\usepackage[numbers]{natbib}
\usepackage{amsthm,amsmath}
\usepackage{graphics,ifthen}
\usepackage{latexsym}
\usepackage{mathrsfs}
\usepackage{amssymb}

\startlocaldefs

\newtheoremstyle{example}{\topsep}{\topsep}%
     {}
     {}
     {\rmfamily}
     {}
     {\newline}
     {\thmname{#1}\thmnumber{ #2}\thmnote{ #3}}

   \theoremstyle{example}

\newcommand{\indic}{\mathbb{I}}

\numberwithin{equation}{section}
\theoremstyle{plain}
\newtheorem{thm}{Theorem}[section]

\newtheorem{prop}{Proposition}[section]

\newtheorem{rem}{Remark}[section]

\newcommand{\Lower}[2]{\smash{\lower #1 \hbox{#2}}}
\newcommand{\ben}{\begin{enumerate}}
\newcommand{\een}{\end{enumerate}}
\newcommand{\bi}{\begin{itemize}}
\newcommand{\ei}{\end{itemize}}

\newcommand{\DoubleRLarrow}[1]{\Lower{-0.01in}{$\underleftarrow{\Lower{0.07in}{$\overrightarrow{\vspace*{0.15in}\hspace*{#1}}$}}$}}

\newcommand{\SingleRarrow}[1]{\Lower{-0.08in}{$\underrightarrow{\hspace*{#1}}$}}

\endlocaldefs

\begin{document}

\begin{frontmatter}
\title{Coag-Frag duality for a class of stable Poisson-Kingman mixtures\protect} \runtitle{Coag/Frag Duality}

\begin{aug}
\author{\fnms{Lancelot F.} \snm{James}\thanksref{t1}\ead[label=e1]{lancelot@ust.hk}},

\thankstext{t1}{Supported in
part by the grant RGC-HKUST 600907 of the HKSAR.}

\runauthor{Lancelot F. James}

\affiliation{Hong Kong University of Science and Technology}

\address[a]{Lancelot F. James,\\ The Hong Kong University of Science and
Technology, \\Department of Information Systems, Business Statistics and Operations Management,\\
Clear Water Bay, Kowloon, Hong Kong.\\ \printead{e1}.}

\contributor{James, Lancelot F.}{Hong Kong University of Science
and Technology}

\end{aug}

\begin{abstract}
Exchangeable sequences of random probability measures (partitions of mass) and their corresponding exchangeable bridges
play an important role in a variety of
areas in probability, statistics and related areas, including
Bayesian statistics, physics, finance and machine learning. An area of theoretical as well as practical interest, is the study of coagulation and fragmentation operators on partitions of mass. In this regard, an interesting but formidable question is the identification of operators and distributional families on mass partitions that exhibit interesting duality relations.  In this paper we identify duality relations for a large sub-class of mixed Poisson-Kingman models generated by a stable subordinator.
Our results are natural generalizations of
the duality relations developed in Pitman~\cite{Pit99}, Bertoin
and Goldschmidt~\cite{BertoinGoldschmidt2004}, and Dong,
Goldschmidt and Martin~\cite{Dong2006}, for the two-parameter Poisson Dirichlet family. These results are deduced from results for corresponding bridges.
\end{abstract}




\end{frontmatter}
\section{Introduction}Exchangeable sequences of random probabilities living in the
space $\mathcal{P}=\{\mathbf{s}=(s_{1},s_{2},\ldots):s_{1}\ge
s_{2}\ge\cdots\ge 0 {\mbox { and }} \sum_{i=1}^{\infty}s_{i}=1\},$
and corresponding exchangeable random probability measures on
$[0,1],$ defined as
\begin{equation}
P(p)=\sum_{k=1}^{\infty}P_{i}\indic_{(U_{i}\leq p)},
\label{bridge}
\end{equation}
where $(U_{i})$ are iid Uniform$[0,1]$ variables independent of
$(P_{i})\in \mathcal{P} ,$ play an important role in a variety of
areas in probability, statistics and related areas, including
Bayesian statistics, physics, finance and machine learning. Some
references, are as
follows~\cite{Chatterjee, Ferg73, Feng,
Diaconis,IJ2001, JLP,Ruelle,
Kerov, Kingman75, Pit96}. Our primary references in this paper
will center around applications to coagulation/fragmentation
phenomena. For a general summary of some of these applications,
and for the concepts and notations we use in this exposition, we
refer to the monographs~\cite{BerFrag, Pit06}, and
also~\cite{BerLegall03}.

One of the most interesting examples in the literature is the
two-parameter Poisson-Dirichlet family of laws on $\mathcal{P},$
say $\mathrm{PD}(\alpha,\theta),$ indexed by $0\leq \alpha<1$ and
$\theta>-\alpha,$ as discussed in \cite{PY97}. The corresponding
$\mathrm{PD}(\alpha,\theta)$-bridge, denoted as
$P_{\alpha,\theta}(p),$ is the random distribution function
defined by setting $(P_{i})\sim PD(\alpha,\theta)$
in~(\ref{bridge}). The Poisson-Dirichlet $(\alpha,\theta)$ family
arises in connection with the lengths of excursions of Bessel
processes and often appear, in some guise, in the study of
phenomena involving positive $\alpha$-stable subordinators and/or
gamma subordinators. These processes also play an important role
in Bayesian statistics and machine learning. See
Bertoin~\cite{BerFrag} for applications to
coagulation/fragmentation phenomena and Ishwaran and
James~\cite{IJ2001} [see also Pitman~\cite{Pit96}]for applications
to Bayesian statistics, where in particular $P_{\alpha,\theta}$ is
referred to as a Pitman-Yor process. Under this name the process
has also been applied to problems arising in natural language
processing, see for instance \cite{Teh, Wood1,Wood2}. In fact as
shown explicitly in \cite{Wood1}, these methods are working with
coagulation/fragmentation operations at the level of the Poisson
Dirichlet random probability measures (bridges). They show these connections
lead to a significant reduction in the complexity of an
$\infty$-gram natural language model. When $\theta>0$ and
$\alpha=0$ $P_{0,\theta}$ is a Dirichlet process made popular by
Ferguson~\cite{Ferg73}.

In regards to general $(P_{i})\in {\mathcal P}$ an interesting
question arising in the study of coagulation and fragmentation
processes~\cite{BerFrag, Pit06} is as follows. For $X,Y$ random
exchangeable sequences in $\mathcal{P},$ describe in an
informative way the conditional distribution of $X|Y$ and $Y|X.$
Naturally $X$ and $Y$ should also have some interesting
interpretations. We also note that it is not necessarily the case
that both laws $X$ and $Y$ are initially known. This is the
essence of what is known as a coagulation-fragmentation duality,
and is generally a difficult problem. Generically this duality can be read using the following diagram
for $X,$ $Y$ in ${\mathcal P},$
\begin{eqnarray*}
  & X|Y &  \nonumber\\
 Y & \DoubleRLarrow{0.5in} & X \\
  & Y|X &\nonumber
\end{eqnarray*}

Pitman~\cite{Pit99} was able
to derive a remarkable duality formula for certain members of the
$\mathrm{PD}(\alpha,\theta)$ family, where in particular he
describes the relationships between $X\sim
\mathrm{PD}(\alpha\delta,\theta)$ and
$Y\sim\mathrm{PD}(\alpha,\theta)$ for $0\leq \delta<1.$ This
relationship acts in a multiplicative fashion on the first
component. The coagulation/fragmentation duality in
Pitman~\cite{Pit99} may be described in terms of the following
diagram as given in~\cite{Pit06}; for $0< \alpha<1, 0\leq
\delta<1, \theta>-\alpha\delta$,
\begin{eqnarray}
  & \mathrm{PD}(\delta,\frac{\theta}{\alpha})-\mathrm{Coag} &  \nonumber\\
  \mathrm{PD}(\alpha,\theta) & \DoubleRLarrow{0.5in} & \mathrm{PD}(\alpha\delta,\theta) \label{Pitman}\\
  & \mathrm{PD}(\alpha,-\alpha\delta)-\mathrm{Frag} &\nonumber
\end{eqnarray}
More recently, using the $\mathrm{PD}(0,\theta)$
family, Bertoin and Goldschmidt~\cite{BertoinGoldschmidt2004}
describe an additive duality relationship where $X\sim
\mathrm{PD}(0,\theta)$ and $Y\sim\mathrm{PD}(0,1+\theta).$ This
additive duality is generalized to the
$\mathrm{PD}(\alpha,\theta)$ family in Dong, Goldschmidt, and
Martin(DGM)~\cite{Dong2006}. Their results can be represented as follows, for $\theta>-\alpha,$ and $0\leq \alpha<1,$
\begin{eqnarray}
  & \beta_{(\frac{(1-\alpha)}{\alpha},\frac{(\theta+\alpha)}{\alpha})}-\mathrm{Coag} &  \nonumber\\
  \mathrm{PD}(\alpha,1+\theta) & \DoubleRLarrow{0.5in} & \mathrm{PD}(\alpha,\theta) \label{DGMd}\\
  & \mathrm{Frag}-\mathrm{PD}(\alpha,1-\alpha) &\nonumber
\end{eqnarray}
We will give a precise meaning of the Coag/Frag operators later.

In general, it is not clear how one
can obtain similar results for other $(\alpha,\theta)$ parameters
values or other families in $\mathcal{P}.$ In this paper we, using results we develop for bridges, identify a large class of laws on $\mathcal{P}$ where explicit duality relations exist. These can be seen as natural extensions of the results in~\cite{Pit99, BertoinGoldschmidt2004, Dong2006}. The class represents a sub-class of Poisson-Kingman mixtures generated by stable subordinators that we denote as having laws $\mathbb{P}_{\alpha}(\zeta),$ where $\zeta$ denotes a non-negative random variable. We describe more details of this class as well as relevant result for more general processes in the next section.
\section{Exchangeable bridges and partitions}
Following Bertoin~\cite[Definition 2.1, p.67]{BerFrag},(see also
Pitman\cite[section 5]{Pit06}), an infinite numerical sequence
$\mathbf{s}=(s_{1},s_{2},\ldots)$ is said to be a
\textbf{mass-partition} if $\mathbf{s}$ is an element of the
space,
$$
\mathcal{P}_{\mathrm{m}}=\{\mathbf{s}=(s_{1},s_{2},\ldots):s_{1}\ge
s_{2}\ge\cdots\ge 0 {\mbox { and }} \sum_{i=1}^{\infty}s_{i}\leq
1\}.
$$

The quantity
$$
s_{0}:=1-\sum_{i=1}^{\infty}s_{i},
$$
which may be $0,$ is referred to as the \emph{total mass of dust}.
From Bertoin~(\cite{BerFrag}, Definition 4.6, p. 191), a random
caglad process $b_{\mathbf{s}}$ on $[0,1]$ is said to be an
\textbf{s-bridge} if it is distributed as
$$
b_{\mathbf{s}}(y)=s_{0}y+\sum_{k=1}^{\infty}s_{i}\indic_{(U_{i}\leq
y)}, y\in[0,1],
$$
for $(U_{i})$  a sequence of iid Uniform$[0,1]$ random variables.
If $\mathbf{s}\sim\mathbb{P}$, i.e. if
$\mathbf{s}$ is randomized according to some law $\mathbb{P},$
then $b_{\mathbf{s}}$ is said to be a $\mathbb{P}$-bridge. It
follows that $\mathcal{P}$ is a subspace of
$\mathcal{P}_{\mathrm{m}}$ such that $\sum_{i=1}^{\infty}s_{i}=1.$
Furthermore, for all $\mathbf{s}\in P_{\mathbf{m}},$
$\mathrm{Rank}(s_{0}, s)\in \mathcal{P}.$ Hence we see that the random probability measure in~(\ref{bridge})
is a $\mathbb{P}$-bridge , with $(s_{i})\overset{d}=(P_{i})\in \mathcal{P}$ distributed according to some law $\mathbb{P}$ with $s_{0}=0.$ An important property,  which we shall exploit, is that the law of the $\mathbb{P}$-bridge is in bijection to the law of the sequence $(P_{i})\sim \mathbb{P}.$ Additionally let
$$
b^{-1}_{\mathbf{s}}(r)=\inf\{v\in [0,1]: b_{\mathbf{s}}(r)>r\},
r\in[0,1]
$$
denote the right continuous inverse of the bridge. Equivalently
this is a random quantile function.   An
exchangeable partition of $[n]=\{1,2,\ldots,n\}$ generated from an exchangeable
bridge, say $b_{\mathbf{s}},$ can be obtained by the equivalence relations
$$
i\sim j{\mbox { iff }}b^{-1}_{\mathbf{s}}(U'_{i})=b^{-1}_{\mathbf{s}}(U'_{j})
$$
based on $n$ iid Uniform$[0,1]$ variables
$(U'_{1},\ldots,U'_{n}).$ An infinite partition, $\Pi$ of $\mathbb{N}$ is formed by
considering a countably infinite set of uniforms.
The distribution of such infinite exchangeable partitions is referred to as an exchangeable partition probability function (EPPF). We see that the random probability measure in~(\ref{bridge})
is a $\mathbb{P}$-bridge , with $(s_{i})\overset{d}=(P_{i})\in \mathcal{P}$ distributed according to some law $\mathbb{P}$ with $s_{0}=0.$ An important property,  which we shall exploit, is that the law of the $\mathbb{P}$-bridge is in bijection to the law of the sequence $(P_{i})\sim \mathbb{P},$ and also to the corresponding EPPF , specifying the law of the exchangeable partition $\Pi$ with ranked frequencies $(P_{i}),$ which we shall refer to as a $\mathbb{P}$-EPPF.
In this manuscript we will also utilize properties of simple bridges. In particular if $\mathbf{s}=(u,0,\ldots)$ is a simple mass-partition then
$b_{u}(y)=(1-u)y+u\indic_{(U_{1}\leq y)}$ is referred to as a
\textbf{simple bridge}. If $u=s_{1}$ is a random variable then one has a randomized simple bridge given by,
\begin{equation}
b_{s_{1}}(y)=s_{0}y+s_{1}\indic_{(U_{1}\leq y)}
\label{simplebridge}
\end{equation}

\subsection{Poisson Kingman distributions determined by a stable subordinator}
Recall from Pitman~\cite{Pit02}, that for $0<\alpha<1$ a sequence $(P_{i})$ has a Poisson-Kingman law generated by a $\alpha$-stable subordinator with mixing distribution $\eta,$ say $\mathrm{PK}_{\alpha}(\eta),$ if its law can be constructed as follows; Let $(J_{i})$ denote the ranked jumps of a stable subordinator such that $T=\sum_{k=1}^{\infty}J_{k}$ is equivalent in distribution to a positive $\alpha$-stable random variable, with density denoted as $f_{\alpha}(t)$ and  whose log Laplace transform is given by $-C\omega^{\alpha}$ for some constant $C>0$ and each $\omega>0.$  Hereafter, due to scaling properties, we can take $C=1.$ Set $(P_{i}=J_{i}/T),$ then it follows that $(P_{i})$ has a $\mathrm{PD}(\alpha,0)$ distribution. Denote by $\mathrm{PD}(\alpha|t)$ the conditional distribution of $(P_{i})|T=t,$ then
$$
\mathrm{PK}_{\alpha}(\eta):=\int_{0}^{\infty}\mathrm{PD}(\alpha|t)\eta(dt)
$$
The $PD(\alpha,\theta)$ laws arises as a special case by choosing $\eta(dt)/dt$ proportional to $t^{-\theta}f_{\alpha}(t),$ which is the density of  a polynomially tilted stable random variable. The classical Poisson-Dirichlet case, $\mathrm{PD}(0,\theta),$ arises by letting $\alpha$ go to zero in an appropriate sense. An important feature of the general $\mathrm{PK}_{\alpha}(\eta)$ class of laws and its limiting cases, is that as shown in \cite{Pit02,Pit06,GnedinPitmanI}, see for instance Pitman\cite[Theorem 8, p.14]{Pit02}, that these are the only cases where the EPPF of an infinite exchangeable random partition $\Pi$ with ranked frequencies $(P_{i})$ has Gibbs form. Additionally, we will make use of the following fact, if $T$ is a random variable with distribution $\eta,$ then, from Pitman\cite[Proposition 13, p.20]{Pit02}, $S=T^{-\alpha}$ is the $\alpha$-DIVERSITY of the $\mathrm{PK}_{\alpha}(\eta)$ partition. That is, if $K_{n}$ denotes the number of distinct blocks of a $\mathrm{PK}_{\alpha}(\eta)$-EPPF partition of $[n],$ then $K_{n}/n^{\alpha}$ converges almost surely to $S$ as $n$ converges to $\infty,$  and almost surely,
$$
T=S^{-1/\alpha}:=\lim_{i\rightarrow \infty}{(i\Gamma(1-\alpha)P_{i})}^{-1/\alpha}
$$
In other words $S$ and $T$ are completely determined by the corresponding $(P_{i})$ sequence.
\subsection{The $\mathbb{P}_{\alpha}(\zeta)$ family}
As mentioned in the introduction, in this paper we show that one can extend the results of Pitman~\cite{Pit99} and Bertoin and Goldschmidt~\cite{BertoinGoldschmidt2004}, Dong, Goldschmidt, and
Martin(DGM)~\cite{Dong2006} to a large class of processes whose $\mathbb{P}$ law is given by $\mathrm{PK}_{\alpha}(\eta^{*}),$ where $\eta^{*}$ belongs to a class of mixing distributions corresponding to random variables of the form,
$$
T\overset{d}=\frac{\tau_{\alpha}(\zeta)}{\zeta^{1/\alpha}}.
$$
$\zeta$ is a non-negative random variable taken independent of  $(\tau_{\alpha}(s),s>0),$ which is a generalized gamma subordinator whose L\'evy exponent, i.e. its -log Laplace transform of $\tau_{\alpha}(1),$ is given by
\begin{equation}
\psi_{\alpha}(\omega)=(1+\omega)^{\alpha}-1
\label{GGexp}
\end{equation}
for $\omega>0.$ The conditional density of $T|\zeta$ is given by
$$
f_{\alpha}(s){\mbox e}^{-(s\zeta^{1/\alpha}-\zeta)}
$$
and the hence the density of $T$ can be expressed as,
$$
\eta^{*}(ds)/ds=f_{\alpha}(s)\int_{0}^{\infty}{\mbox e}^{-(sy^{1/\alpha}-y)}F_{\zeta}(dy)=f_{\alpha}(s)\mathbb{E}[{\mbox e}^{-(s\zeta^{1/\alpha}-\zeta)}]
$$
where $F_{\zeta}$ denotes the distribution function of $\zeta.$ Hence, if $\zeta$ is random, a conditional distribution of $\zeta|T=s$ is specified by
\begin{equation}
F_{\zeta,\alpha}(dy|s)\propto {\mbox e}^{-(sy^{1/\alpha}-y)}F_{\zeta}(dy).
\label{condzeta}
\end{equation}
It follows that, for fixed $\alpha,$ the law  of $(P_{i})\sim \mathrm{PK}_{\alpha}(\eta^{*})$ varies according to the distribution of $\zeta,$ and hence we denote this law as $\mathbb{P}_{\alpha}(\zeta):=\mathrm{PK}_{\alpha}(\eta^{*}).$ Importantly, the corresponding $\mathbb{P}_{\alpha}(\zeta)$-bridge can be written as
\begin{equation}
Q_{\alpha,\zeta}(y)\overset{d}=\frac{\tau_{\alpha}(\zeta y)}{\tau_{\alpha}(\zeta)}\overset{d}=\sum_{i=1}^{\infty}P_{i}\indic_{(U_{i}\leq
y)},  y\in[0,1]
\label{Laplacerpm}
\end{equation}
where, $(P_{i})\sim \mathbb{P}_{\alpha}(\zeta).$

This construction of $\mathbb{P}_{\alpha}(\zeta)$ laws coincides with random processes discussed in Pitman and Yor \cite[p. 877-878]{PY97}, which is used to prove  Pitman and Yor \cite[Proposition 21, p. 869]{PY97}. Proposition 21 of that work shows that if $\zeta\overset{d}=\gamma_{\theta/\alpha}$ where $\gamma_{\theta/\alpha}$ denotes a random variable with a gamma distribution with shape parameter $(\theta/\alpha),$ and scale $1,$  then $\mathbb{P}_{\alpha}(\gamma_{\theta/\alpha})=PD(\alpha,\theta)$ for $\theta>0.$ Note this does not include the case of $\mathrm{PD}(\alpha,\theta)$ for $-\alpha<\theta<0.$ However $\mathbb{P}_{\alpha}(0)=\mathrm{PD}(\alpha,0).$ Furthermore when $\zeta=b$ is a positive constant, $\mathbb{P}_{\alpha}(b)$ corresponds to the case of the Poisson-Kingman model determined by the generalized gamma subordinator as described in Pitman \cite[Section 5.2]{Pit02}. In this case the bridge $Q_{\alpha,b}$ has been studied from a Bayesian perspective in \cite{James2002,LMP1,LMP2, JLP2}. However it is evident that, due to the generality of $\zeta$, the class of $\mathbb{P}_{\alpha}(\zeta)$  laws is significantly larger than the special cases mentioned.

In order to establish our results we will work directly with $\mathbb{P}_{\alpha}(\zeta)$-bridges, $Q_{\alpha,\zeta}$.  In fact, we will show that working with $Q_{\alpha,\zeta}$ is rather transparent in terms of identifying which laws on $\mathcal{P}$ are related in the sense of Coag-Frag operators. While the operators we discuss are of similar type to those in \cite{Pit99, BertoinGoldschmidt2004, Dong2006}, we cannot rely on the fine properties of the $\mathrm{PD}(\alpha,\theta)$ family utilized by those authors. For example,
one can show that the coagulation operators in \cite{Pit99, BertoinGoldschmidt2004, Dong2006} are in bijection to the operation of composition of independent bridges. The dual relationship between compositions of independent bridges and coagulation operations can be found in the works of Bertoin and Le Gall, \cite{BerFrag, BerLegall00,
BerLegall03} and Pitman~\citep[Lemma 5.18]{Pit06}. Our coagulation operations will be defined via the compositions of generally dependent bridges. Nonetheless, for a given input sequence $(p_{i}),$ we are able to give good descriptions of the conditional distribution of the relevant coagulation operator, which as we shall show reduce to conditional distributions given the DIVERSITY or local time determined by the input sequence. We will also show that the dual fragmentation operators are exactly the same as those used in \cite{Pit99, BertoinGoldschmidt2004, Dong2006}, where, in contrast to the coagulations operators, our inputs are indeed independent of the respective $\mathrm{PD}(\alpha,-\alpha\delta)$ and $\mathrm{PD}(\alpha,1-\alpha)$ fragmenting variables in $\mathcal{P}.$
\section{Pitman style coagulation and fragmentation operations for $\mathbb{P}_{\alpha}(\zeta).$}
In order to establish an analogue of (\ref{Pitman}) for the $\mathbb{P}_{\alpha}(\zeta)$ family of laws we first identify an appropriate coagulation operation. As we mentioned in the previous section there is a close relationship between the notion of coagulation operators on $\mathcal{P},$ or in terms of corresponding exchangeable partitions, and the idea of composition of bridges, say $$
F_{1}(y)=\sum_{k=1}^{\infty}P^{(1)}_{k}\indic_{(V_{k}\leq y)}{\mbox { and }} F_{2}(y)=\sum_{k=1}^{\infty}P^{(2)}_{k}\indic_{(U_{k}\leq y)},$$
where $(V_{k})$ and $(U_{k})$ are independent sequences of iid uniform variables, and independent of these, ${(P^{(1)}_{k}})$ and $(P^{(2)}_{k})$ have marginal laws on $\mathcal{P}$ denoted as $\mathbb{P}^{(1)}$ and $\mathbb{ P}^{(2)}.$
For instance, following Bertoin~\cite[Section 4]{BerFrag} a coagulation operation on partitions can be defined by the composition
$$F_{2}(F_{1}(y)):=F_{2}\circ F_{1}(y)=\sum_{k=1}^{\infty}P^{(2)}_{k}\indic_{(F_{1}^{-1}(U_{k})\leq y)}$$ in terms of partitions induced by the relation
\begin{equation}
i\sim j{\mbox { iff }}F_{1}^{-1}\circ F^{-1}_{2}(U'_{i})=F_{1}^{-1}\circ F^{-1}_{2}(U'_{j})
\label{rel1}
\end{equation}
based on $n$ iid Uniform$[0,1]$ variables
$(U'_{1},\ldots,U'_{n}).$ If viewed in stages, one first creates a partition following the law of the $\mathbb{P}^{(2)}$-EPPF associated with $F_{2}$
by the relationship
$$
i\sim j{\mbox { iff }}F^{-1}_{2}(U'_{i})=F^{-1}_{2}(U'_{j}).
$$
Given the partition of $[n],$ say $\{B_{1},\ldots,B_{K^{(2)}_{n}}\},$ induced by this operation with $K^{(2)}_{n}=k$ unique blocks, there are $U^{*}_{1},\ldots,U^{*}_{k}$ distinct iid uniform variables associated with the $k$ blocks with labels $\{1,\ldots,k\}.$ The blocks are further merged by the relation, i.e. merge $B_{i}$ and $B_{j}$ according to,
$$
i\sim j{\mbox { iff }}F_{1}^{-1}(U^{*}_{i})=F_{1}^{-1}(U^{*}_{j})
$$
From Pitman\cite[Section 5, Lemma5.18]{Pit06} the corresponding  Coag operator on $\mathcal{P},$ which includes the Coag operator in (\ref{Pitman}) is defined as follows.
Let $(I^{\mathbb{P}^{(1)}}_{j})$ denote the interval partition of
$\mathbb{P}^{(1)}$ as described in~\citep[p. 111]{Pit06} induced by a $\mathbb{P}^{(1)}$-bridge then for $(P^{(2)}_{i})\sim \mathbb{P}^{(2)},$ it follows that
$$
\mathrm{Rank}\left(\sum_{i=1}^{\infty}P^{(2)}_{i}\indic_{(U_{i}\in
I^{\mathbb{P}^{(1)}}_{j})}, j\ge 1\right)
$$
is equivalent in distribution to the sequence in $\mathcal{P}$ induced by the composition of bridges $F_{2}\circ F_{1}.$ Hence under these specifications, setting $(P^{(2)}_{i})=(p_{i})$, the Coag operator
$(\mathbb{P}^{(1)}-\mathrm{Coag})((p_{i}),\cdot)$
is the distribution of
$$
\mathrm{Rank}\left(\sum_{i=1}^{\infty}p_{i}\indic_{(U_{i}\in
I^{\mathbb{P}^{(1)}}_{j})}, j\ge 1\right).
$$
In the literature it is usually assumed that the sequences $(P^{(1)}_{i}),(P^{(2)}_{i})$ are independent, which would mean that the interval partition $(I^{\mathbb{P}^{(1)}}_{j})$ is independent of $(P^{(2)}_{i}).$ In terms of the relation (\ref{rel1}) this means that the merging of blocks in the second stage only depends on the number of blocks $K^{(2)}_{n}=k$ and is otherwise conducted independently with respect to a $\mathbb{P}^{(2)}$-EPPF. This is case for the operator defined in (\ref{Pitman}). However it is clear,
working with the explicit constructions of $P_{1}$ and $P_{2},$ and using the relation (\ref{rel1}), that
$(\mathbb{P}^{(1)}-\mathrm{Coag})((p_{i}),\cdot),$ coagulation operators induced by possibly dependent sequences $(P^{(1)}_{i}),(P^{(2)}_{i})$ still makes sense except now its distribution is a bit more complicated.

We now show that the relevant coagulation operator for the $\mathbb{P}_{\alpha}(\zeta)$ class is of this form, but despite this extra dependence we will be able to show that its distribution can be described quite clearly.
\subsection{Compositions of $\mathbb{P}_{\alpha}(\zeta)$-bridges and resulting Coag operators}
For $0\leq \alpha<1$ and $0\leq \delta<1,$ let $\tau_{\alpha}$ and $\tau_{\delta}$ denote independent generalized gamma subordinators with laws specified by~(\ref{GGexp}), where in the second case $\alpha$ is replaced by $\delta.$   Then for a common random variable $\zeta,$ define bridges
\begin{equation}
Q_{\delta,\zeta}(y)=\frac{\tau_{\delta}(\zeta y)}{\tau_{\delta}(\zeta)}{\mbox { and }}
Q_{\alpha,\tau_{\delta}(\zeta)}(y)=\frac{\tau_{\alpha}(\tau_{\delta}(\zeta) y)}{\tau_{\alpha}(\tau_{\delta}(\zeta))}
\label{Laplacebridge}
\end{equation}
If $\zeta\overset{d}=\gamma_{\theta/(\delta\alpha)},$ then $\tau_{\delta}(\zeta)\overset{d}=\gamma_{\theta/\alpha}$
and it can be deduced from  Pitman and Yor \cite[Proposition 21, p. 869, and p.877-878]{PY97}, that $Q_{\delta,\zeta}$ is a $\mathrm{PD}(\delta,\theta/\alpha)$-bridge and $Q_{\alpha,\tau_{\delta}(\zeta)}$ is a $\mathrm{PD}(\alpha,\theta)$-bridge. When $\zeta=0,$ the bridges reduce to the case of $\mathbb{PD}(\delta,0)$ and $\mathbb{PD}(\alpha,0)$ We can deduce further from Pitman and Yor \cite[p.877-878]{PY97} that these are the only cases where the bridges $Q_{\alpha,\tau_{\delta}(\zeta)}$ and $Q_{\delta,\zeta}$ are independent. Nonetheless, it is obvious by construction that the composition of these bridges yields
\begin{equation}
Q_{\alpha,\tau_{\delta}(\zeta)}(Q_{\delta,\zeta}(y))\overset{d}=Q_{\alpha\delta,\zeta}(y).
\label{Coagcomp}
\end{equation}
Note that we use the fact that $\tau_{\alpha}(\tau_{\delta}(\zeta))\overset{d}=\tau_{\alpha\delta}(\zeta).$ Hence,
this allows us to write
\begin{eqnarray}
  & \mathbb{P}_{\delta}(\zeta)-\mathrm{Coag} &  \nonumber\\
  \mathbb{P}_{\alpha}(\tau_{\delta}(\zeta))& \SingleRarrow{0.5in} &   \mathbb{P}_{\alpha\delta}(\zeta) \label{PitmanA}
\end{eqnarray}
where an initial description of $\mathbb{P}_{\delta}(\zeta)-\mathrm{Coag}$ is given in the next proposition.
\begin{prop}\label{1}Considering the bridges in~(\ref{Laplacebridge}), let $(I^{\mathbb{P}}_{j},j>1)$ for $\mathbb{P}=\mathbb{P}_{\delta}(\zeta),$ denote the interval partition induced by the $\mathbb{P}_{\delta}(\zeta)$-bridge $Q_{\delta,\zeta}.$ Writing
$$
Q_{\alpha,\tau_{\delta}(\zeta)}(y)=\sum_{k=1}^{\infty}P_{k}\indic_{U_{k}\leq y},
$$
it follows that the marginal distribution of the sequence $(P_{k})\sim \mathbb{P}_{\alpha}(\tau_{\delta}(\zeta)),$ but is not in general independent of the $\mathbb{P}_{\delta}(\zeta)$ interval partition $(I^{\mathbb{P}}_{j},j>1).$
\begin{enumerate}
\item[(i)]However, from (\ref{Coagcomp}) it follows that
\begin{equation}
\mathrm{Rank}\left(\sum_{k=1}^{\infty}P_{k}\indic_{(U_{k}\in
I^{\mathbb{P}}_{j})}, j\ge 1\right)\sim \mathbb{P}_{\alpha\delta}(\zeta)
\label{Prank}
\end{equation}
\item[(ii)]Hence setting $(P_{k})=(p_{k}),$ the $\mathbb{P}_{\delta}(\zeta)-\mathrm{Coag}((p_{k}),\cdot)$ is the distribution on
$\mathcal{P}$ of
\begin{equation}
\mathrm{Rank}\left(\sum_{k=1}^{\infty}p_{k}\indic_{(U_{k}\in
I^{\mathbb{P}}_{j})}, j\ge 1\right)
\label{Pcoag}
\end{equation}
where the conditional distribution of the $\mathbb{P}_{\delta}(\zeta)$ interval partition $(I^{\mathbb{P}}_{j},j>1)$ given $(P_{k})=(p_{k})$ is not independent of $(p_{k}),$ and is otherwise determined by the constructions in~(\ref{Laplacebridge}).
\end{enumerate}
\end{prop}
The description of the distribution of the $\mathbb{P}_{\delta}(\zeta)-\mathrm{Coag}$ operator in statement [(ii)] is rather vague. We now provide a much better description. In regards to the bridges defined in~(\ref{Laplacebridge}), set
\begin{equation}
T_{1}=\frac{\tau_{\delta}(\zeta)}{\zeta^{1/\delta}}{\mbox { and }}
T_{2}=\frac{\tau_{\alpha}(\tau_{\delta}(\zeta))}{{[\tau_{\delta}(\zeta)]}^{1/\alpha}}
\label{Laplacebridge2}
\end{equation}
$T^{-\delta}_{1}$ is the $\delta$-DIVERSITY of $\mathbb{P}_{\delta}(\zeta)$ and $T_{2}^{-\alpha}$ is the $\alpha$-DIVERSITY of $\mathbb{P}_{\alpha}(\tau_{\delta}(\zeta)).$ Hence they are completely determined given realizations from the respective $\mathbb{P}_{\delta}(\zeta)$ and $\mathbb{P}_{\alpha}(\tau_{\delta}(\zeta))$ sequences in $\mathcal{P}.$
Now setting $T_{1}=s$ it follows that
\begin{equation}
T_{2}=\frac{\tau_{\alpha}(\zeta^{1/\delta}s)}{{[\zeta^{1/\delta}s]}^{1/\alpha}}{\mbox { and }}
Q_{\alpha,\tau_{\delta}(\zeta)}(y)=\frac{\tau_{\alpha}(\zeta^{1/\delta}sy)}{\tau_{\alpha}(\zeta^{1/\delta}s)}
\label{Laplacebridge3}
\end{equation}
Applying Bayes rule a conditional density of $T_{1}|T_{2}=v,\zeta$ is given by
$$
f_{1}(s|v,\zeta)\propto f_{\delta}(s){\mbox e}^{-v\zeta^{1/(\alpha\delta)}s^{1/\alpha}}
$$
A conditional density of $T_{2}|\zeta,$ is
$$
f_{2}(v|\zeta)=f_{\alpha}(v)\int_{0}^{\infty}{\mbox e}^{-v\zeta^{1/(\alpha\delta)}s^{1/\alpha}}{\mbox e}^{\zeta}f_{\delta}(s)ds
$$
Hence it follows that a conditional density of $T_{1}|T_{2}=v$ is given by
\begin{equation}
\eta^{(v)}(ds)/ds\propto f_{\delta}(s)\mathbb{E}[{\mbox e}^{-v\zeta^{1/(\alpha\delta)}s^{1/\alpha}}{\mbox e}^{\zeta}]
\label{PKmix}
\end{equation}
Now using Pitman and Yor \cite[p.877-878]{PY97}, gives the following result.
\begin{thm}\label{coag1} Consider the setting in Proposition~\ref{1}, with dependent bridges defined by (\ref{Laplacebridge}), and the associated variables $T_{1}$ and $T_{2}$ defined by (\ref{Laplacebridge2}). Then, for the sequence $(P_{k})$ whose marginal follows a $\mathbb{P}_{\alpha}(\tau_{\delta}(\zeta))$ distribution, set $P_{k}=(p^{(v)}_{k})$, where this indicates that the particular realization $(p^{(v)}_{k})$ corresponds to $T_{2}=v.$ Then the distribution of the $\mathbb{P}_{\delta}(\zeta)-\mathrm{Coag}((p^{(v)}_{k}),\cdot)$ given $P_{k}=(p^{(v)}_{k})$ is equivalent to the distribution of
\begin{equation}
\mathrm{Rank}\left(\sum_{k=1}^{\infty}p^{(v)}_{k}\indic_{(U_{k}\in
I^{\mathbb{Q}^{(v)}}_{j})}, j\ge 1\right)
\label{Pcoag}
\end{equation}
where for fixed $(p^{(v)}_{k}),$ $(I^{\mathbb{Q}^{(v)}}_{j}), j\ge 1)$ is equivalent in distribution to a  $\mathbb{Q}^{(v)}$ interval, with
$$
\mathbb{Q}^{(v)}=\mathrm{PK}_{\delta}(\eta^{(v)}):=\int_{0}^{\infty}\mathrm{PD}(\delta|s)\eta^{(v)}(ds).$$
That is, the conditional distribution of the $\mathbb{P}_{\delta}(\zeta)$ interval partition given $(p^{(v)}_{k})$ only depends on $T_{2},$  and equates with the interval partition of a Poisson-Kingman law generated by a $\delta$-stable subordinator with mixing distribution $\eta^{(v)}$ defined in~(\ref{PKmix}). Equivalently the conditional distribution of the marginally $\mathbb{P}_{\delta}(\zeta)$-bridge constructed in~$(\ref{Laplacebridge}),$ given $(p^{(v)}_{k}),$ is equivalent to a $\mathrm{PK}_{\delta}(\eta^{(v)})$-bridge.
\end{thm}
\begin{proof}
Noting (\ref{Laplacebridge3}), it follows from Pitman and Yor~\cite[p. 877, see eq. (96) and (97)]{PY97},
that the bridges $Q_{\delta,\zeta},Q_{\alpha,\tau_{\delta}(\zeta)},$ are conditionally independent given $T_{1}=s$ and $T_{2}=v,$ and $\zeta,$  and have $\mathrm{PD}(\delta|s)$ and $\mathrm{PD}(\alpha|v)$ distributions respectively. Hence the conditional distribution of the bridge $Q_{\delta,\zeta}$ given $(P_{k})=(p^{(v)}_{k}),$ equates with the conditional distribution of $Q_{\delta,\zeta}$ given $T_{2}=v.$ Which is obtained by finding the conditional density of $T_{1}|T_{2}=v.$
\end{proof}

In the next result, we show that the construction of the bridges in (\ref{Laplacebridge}), and the results discussed in Pitman and Yor~\cite[p. 877]{PY97}, identifies a coagulation operation expressed in terms of conditionally independent processes.

\begin{thm}Consider the bridges defined by (\ref{Laplacebridge}), and the associated variables $T_{1}$ and $T_{2}$ defined by (\ref{Laplacebridge2}). Then,
\begin{enumerate}
\item[(i)]conditional on $T_{1}=s,$ the bridges  $Q_{\delta,\zeta}$ and $Q_{\alpha,\tau_{\delta}(\zeta)}$ are conditionally independent.
\item[(ii)]In particular, given $T_{1}=s,$ $Q_{\delta,\zeta}$ has the distribution of $\mathrm{PD}(\delta|T_{1}=s)$-bridge not depending on $\zeta.$
\item[(iii)]Conditional on $T_{1}=s$ and $\zeta=b$ $Q_{\alpha,\tau_{\delta}(\zeta)}$ is a $\mathbb{P}_{\alpha}(b^{1/\delta}s)$-bridge. That is to say a generalized gamma bridge.
\item[(iv)]Conditional on $T_{1}=s,$  $Q_{\alpha,\tau_{\delta}(\zeta)}$ is a $\mathbb{P}_{\alpha}(\zeta^{1/\delta})$-bridge. Where the law of $\zeta$ depends conditionally on $T_{1}=s,$ and is specified by $F_{\zeta,\delta}(\cdot|s)$ defined in \ref{condzeta}.
\item[(v)]In reference to the $\mathbb{P}_{\delta}(\zeta)-\mathrm{Coag}((p_{k}),\cdot)$ defined by \ref{Pcoag}, it  follows that conditional on $T_{1}=s, and (P_{k})=(p_{k})$ the distribution of the $\mathbb{P}_{\delta}(\zeta)$ interval partition $(I^{\mathbb{P}}_{j},j>1)$ does not depend on $(P_{k})$ and is equivalent in distribution to a $\mathrm{PD}(\delta|T_{1}=s)$  interval partition. Conditional on $T_{1}=s,$ the sequence $(P_{k})$ follows a generalized gamma law  $\mathbb{P}_{\alpha}(\zeta^{1/\delta}s)$
    \end{enumerate}
\end{thm}
\begin{proof}Noting (\ref{Laplacebridge3}), it follows that the bridge $Q_{\alpha,\tau_{\delta}(\zeta)}$ can be expressed in terms of some function of the variables $(\tau_{\alpha},T_{1},\zeta)$ where $\tau_{\alpha}$ is independent of the pair $(T_{1},\zeta)$ and also $Q_{\delta,\zeta}.$ From Pitman and Yor~\cite[p. 877, see eq. (96) and (97)]{PY97}, it follows that $Q_{\delta,\zeta}$ conditioned on $T_{1}=s$ is conditionally independent of $\zeta,$ and has the law of a $\mathrm{PD}(\delta|T_{1}=s)$-bridge. These points establish statements [(i)] and [(ii)]. Statements [(iii)] and [(iv)] easily follow from the explicit construction of  $Q_{\alpha,\tau_{\delta}(\zeta)}$ given in  (\ref{Laplacebridge3}). Statement [(v)] follows as a consequence of statements [(i)] to [(iv)].
\end{proof}

The result shows that by conditioning on $T_{1}=s,$ where $T^{-\delta}_{1}$ is the $\delta-$DIVERSITY corresponding to $\mathbb{P}_{\delta}(\zeta),$ that the composition of dependent bridges described in~(\ref{Coagcomp}), can be first expressed in terms of the composition of conditionally independent bridges, all of which depend on a parameter $s.$ Call a bridge a  $\mathbb{P}^{(s)}_{\alpha,\delta}(\zeta)$-bridge if its law is equivalent in distribution to the conditional distribution of $Q_{\alpha,\tau_{\delta}(\zeta)}\circ Q_{\delta,\zeta}$ given $T_{1}=s.$ Then there is the following relation,
\begin{eqnarray}
  & \mathrm{PD}(\delta|s)-\mathrm{Coag} &  \nonumber\\
  \mathbb{P}_{\alpha}(\zeta^{1/\delta}s)& \SingleRarrow{0.5in} &   \mathbb{P}^{(s)}_{\alpha,\delta}(\zeta) \label{PitmanB}
\end{eqnarray}
where relative to (\ref{PitmanB}), for an input $(P_{k})=(p^{(s)}_{k})$ from a $\mathbb{P}_{\alpha}(\zeta^{1/\delta}s)$ sequence in $\mathcal{P}$ the distribution of the Coag operator $\mathrm{PD}(\delta|s)-\mathrm{Coag}((p^{(s)}_{k}),\cdot)$ is equivalent to
the distribution of
$$
\mathrm{Rank}\left(\sum_{i=1}^{\infty}p^{(s)}_{k}\indic_{(U_{k}\in
I^{\mathrm{PD}(\delta|s)}_{j})}, j\ge 1\right).
$$
where now $(I^{\mathrm{PD}(\delta|s)}_{j})$ denotes a $\mathrm{PD}(\delta|s)$ interval partition that is independent of the input sequence $(P_{k})=(p^{(s)}_{k})$ but otherwise they depend on a common parameter $s.$ It follows that the relation (\ref{PitmanA}) arises from (\ref{PitmanB})  by randomizing $s^{-\delta}$  according to the law of the $\delta-$DIVERSITY of $\mathbb{P}_{\delta}(\zeta).$ The diagram in~(\ref{PitmanA}) can hence be expressed in terms of random partitions on $[n]$ as follows.
Step 1. Draw a variable $S$ having the law of the  $\delta-$DIVERSITY of a $\mathbb{P}_{\delta}(\zeta)$ exchangeable partition. Step 2. Setting $S^{-1/\delta}=s,$ form a random partition $\{B_{1},\ldots,B_{K_{n}}\}$  of $[n]$ according to a  $\mathbb{P}_{\alpha}(\zeta^{1/\delta}s)$-EPPF. Step 3. Merge these $K_{n}$ blocks according to an independent $\mathrm{PD}(\delta|s)$-EPPF.
This scheme produces a random partition of $[n]$ according to a  $\mathbb{P}_{\alpha\delta}(\zeta)$-EPPF.
\subsection{Fragmentation}
From Pitman~\cite[p.112]{Pit06}, for an input $(P_{i})=(p_{i})$ a fragmentation operator $\mathbb{P}-\mathrm{Frag}((p_{i}),\cdot)$ is defined as the distribution of
$$
\mathrm{Rank}(p_{i}Q_{i,j},i,j\ge 1).
$$
where $(Q_{i,j})_{j\ge 1}$ has distribution $\mathbb{P}$ for each $i,$ and these sequences are independent as $i$ varies. In other words one splits  the input $(P_{i})$ multiplying each term by an independent sequence of elements in $\mathcal{P}$ having common law $\mathbb{P}.$ In the case of $(\ref{Pitman}),$ the input has a $\mathrm{PD}(\alpha\delta,\theta)$ independent of the $(Q_{i,j})_{j\ge 1}$ having common law $\mathbb{P}=\mathrm{PD}(\alpha,-\alpha\delta).$ In this section we will show that the same  fragmentation operator  $\mathrm{PD}(\alpha,-\alpha\delta)-\mathrm{Frag}((p_{i}),\cdot),$ applied to independent inputs $(P_{i})$ having law $\mathbb{P}_{\alpha\delta}(\zeta)$ gives the coagulation fragmentation duality for the $\mathbb{P}_{\alpha}(\zeta)$ class that generalizes (\ref{Pitman}).
Again we note that, unlike the coagulation operators discussed in the previous section, the input $(P_{i})$ is independent of the $(Q_{i,j})_{j\ge 1},$ which agrees with the formulation in \cite{Pit99}.  Nonetheless the validity of such results is not immediately obvious. In order to do this we first express these fragmentation operations in terms of an equivalent distributional relationship involving bridges. In particular, let $(P^{(k)}_{\alpha,-\alpha\delta})$ denote a collection of independent $\mathrm{PD}(\alpha,-\alpha\delta)-$bridges. Then it is known that the fragmentation results in \cite{Pit99} can be read in terms of the distributional equivalence of bridges, for $y$ in $[0,1],$
$$
P_{\alpha,\theta}(y)\overset{d}=\sum_{k=1}^{\infty}P_{k}P^{(k)}_{\alpha,-\alpha\delta}(y)
$$
where $(P_{k})$ follows a $\mathrm{PD}(\alpha\delta,\theta),$ distribution. The next result extends this to our setting.
\begin{thm}\label{frag1} Let $(P_{k})$ have law $\mathbb{P}_{\alpha\delta}(\zeta)$ chosen independent of a sequence $(P^{(k)}_{\alpha,-\alpha\delta})$ of independent $\mathrm{PD}(\alpha,-\alpha\delta)-$bridges constructed from the collections of independent $\mathrm{PD}(\alpha,-\alpha\delta)$ sequences
$(Q_{k,j})_{j\ge 1}.$ Then,
\begin{enumerate}
\item[(i)]for $Q_{\alpha,\tau_{\delta}(\zeta)}$ a $\mathbb{P}_{\alpha}(\tau_{\delta}(\zeta))$-bridge there is the distributional equivalence
    \begin{equation}
    Q_{\alpha,\tau_{\delta}(\zeta)}(y)\overset{d}=\sum_{k=1}^{\infty}P_{k}P^{(k)}_{\alpha,-\alpha\delta}(y)
    \label{Fragverify}
    \end{equation}
\item[(ii)]Hence, unconditionally, $\mathrm{PD}(\alpha,-\alpha\delta)-\mathrm{Frag}((P_{i}),\cdot)$ has distribution,
$$
\mathrm{Rank}(P_{i}Q_{i,j},i,j\ge 1)\sim \mathbb{P}_{\alpha}(\tau_{\delta}(\zeta)).
$$
\end{enumerate}
\end{thm}
\begin{proof}Since the bridges are exchangeable, it suffices to verify (\ref{Fragverify}) for some fixed $y.$ For each fixed $y,$ let $H^{(y)}$ denote the distribution function of the random variable $P_{\alpha,-\alpha\delta}(y).$ Let $Q_{\alpha\delta,\zeta}$ denote a $\mathbb{P}_{\alpha\delta}(\zeta)$-bridge, and define the random probability measure $Q^{(y)}=Q_{\alpha\delta,\zeta}\circ H^{(y)},$ i.e.
$$
Q^{(y)}(u)=\sum_{k=1}^{\infty}P_{k}\indic_{(P^{(k)}_{\alpha,-\alpha\delta}(y)\leq u)}
$$
It follows that for each fixed $y,$ (not path-wise), that
$$
\int_{0}^{1}uQ^{(y)}(du)\overset{d}=\sum_{k=1}^{\infty}P_{k}P^{(k)}_{\alpha,-\alpha\delta}(y).
$$
But $Q^{(y)}(u)\overset{d}=\tau_{\alpha\delta}(\zeta H^{(y)}(u))/\tau_{\alpha\delta}(\zeta H^{(y)}(1)),$ where $H^{(y)}(1)=1.$
Hereafter set $\tau_{\alpha\delta}(\zeta H^{(y)}(u)):=\tau^{(y)}(u)$
Now recalling the construction of  $Q_{\alpha,\tau_{\delta}(\zeta)}(y)$ as in (\ref{Laplacebridge}), it follows that (\ref{Fragverify}) is verified by showing that,
\begin{equation}
(\tau_{\alpha}(\tau_{\delta}(\zeta)),\tau_{\alpha}(\tau_{\delta}(\zeta)
y))\overset{d}=(\tau^{(y)}(1),\int_{0}^{1}u\tau^{(y)}(du))
\label{jointcheck}
\end{equation}
This will be done by establishing the equivalences of their joint Laplace transforms at positive points $(\omega_{1},\omega_{2}).$ Notice that
$$
\omega_{1}\tau_{\alpha}(\tau_{\delta}(\zeta))+\omega_{2}\tau_{\alpha}(\tau_{\delta}(\zeta)
y)=\int_{0}^{1}[\omega_{1}+\omega_{2}\indic_{(u\leq y)}]\tau_{\alpha}(\tau_{\delta}(\zeta)du)
$$
Similarly,
$$
\omega_{1}\tau^{(y)}(1)+\omega_{2}\int_{0}^{1}u\tau^{(y)}(du)=\int_{0}^{1}[\omega_{1}+\omega_{2}u]\tau^{(y)}(du)
$$
Conditioning on $\tau_{\delta}(\zeta),$ it follows using standard results for linear functionals of positive L\'evy processes that the -log joint Laplace transform of the the left hand side of (\ref{jointcheck}) is given by
$$
\tau_{\delta}(\zeta)\mathbb{E}[{(1+\omega_{1}+\omega_{2}\indic_{(U\leq y)})}^{\alpha}-1]
$$
yielding, for fixed $\zeta,$
$$
\zeta[{(y(1+\omega_{1}+\omega_{2})^{\alpha}+(1-y)(1+\omega_{1})^{\alpha})}^{\delta}-1].
$$
Conditional on $\zeta,$ the joint -log Laplace transform of the right hand side of~(\ref{jointcheck}) can be expressed as
$$
\zeta\mathbb{E}[{(1+\omega_{1}+\omega_{2}P_{\alpha,-\alpha\delta}(y))}^{\alpha\delta}-1],
$$
but
$$
\omega_{1}+\omega_{2}P_{\alpha,-\alpha\delta}(y)=\int_{0}^{1}[\omega_{1}+\omega_{2}\indic_{(u\leq y)}]P_{\alpha,-\alpha\delta}(du).
$$
Furthermore, it is known from \cite{Vershik}, that for  $M:=\int_{0}^{1}g(u)P_{\alpha,-\alpha\delta}(du),$ for some positive function $g,$ that
$$
\mathbb{E}[{(1+M)}^{\alpha\delta}]={(\mathbb{E}[(1+g(U))^{\alpha}])}^{\delta}.
$$
Setting $g(u)=\omega_{1}+\omega_{2}\indic_{(u\leq y)}$, it follows that
$$
\mathbb{E}[{(1+\omega_{1}+\omega_{2}P_{\alpha,-\alpha\delta}(y))}^{\alpha\delta}]={(y(1+\omega_{1}+\omega_{2})^{\alpha}+(1-y)(1+\omega_{1})^{\alpha})}^{\delta}
$$
concluding the result.
\end{proof}
\subsection{Duality} We can now describe the duality relation in terms of the following diagram
\begin{eqnarray}
  & \mathbb{P}_{\delta}(\zeta)-\mathrm{Coag} &  \nonumber\\
  \mathbb{P}_{\alpha}(\tau_{\delta}(\zeta))& \DoubleRLarrow{0.5in} &   \mathbb{P}_{\alpha\delta}(\zeta) \label{CoagFragI}\\
  & \mathrm{PD}(\alpha,-\alpha\delta)-\mathrm{Frag} &\nonumber
\end{eqnarray}
It follows that for $\theta\ge 0,$ (\ref{Pitman}) arises by choosing $\zeta\overset{d}=\gamma_{\theta/(\alpha\delta)}.$
We close with a formal statement.

\begin{thm}Suppose that $X$ and $Y$ are sequences in ${\mathcal P}.$
Then, using the descriptions in Theorems~\ref{coag1} and~\ref{frag1}, the following statements are equivalent.
\begin{enumerate}
\item[(i)]$X\sim \mathbb{P}_{\alpha\delta}(\zeta)$ and
conditional on $X,$
$Y\overset{d}=\mathrm{PD}(\alpha,-\alpha\delta)-\mathrm{Frag}(X,\cdot).$
\item[(ii)]$Y\sim \mathbb{P}_{\alpha}(\tau_{\delta}(\zeta))$ and conditional on $Y,$
$X\sim
 \mathbb{P}_{\delta}(\zeta)-\mathrm{Coag}(Y,\cdot).$

Where in particular, for $Y=(p^{(v)}_{k}),$ indicating its $\alpha-$DIVERSITY or local time has the value $v^{-\alpha},$
$$
\mathbb{P}_{\delta}(\zeta)-(\mathrm{Coag}((p^{(v)}_{k}),\cdot)
\overset{d}=\mathrm{PK}_{\delta}(\eta^{(v)})-(\mathrm{Coag}((p^{(v)}_{k}),\cdot).
$$
Where on the right hand side the $\mathrm{PK}_{\delta}(\eta^{(v)})$ sequence and the input sequence are conditionally independent.
\end{enumerate}
\end{thm}

\section{DGM type coagulation fragmentation duality for the $\mathbb{P}_{\alpha}(\zeta)$ class}
We now proceed to establish generalizations of the coagulation fragmentation duality, (\ref{DGMd}), described in Bertoin and Goldschmidt~\cite{BertoinGoldschmidt2004}, Dong, Goldschmidt, and
Martin~\cite{Dong2006}. In order for us to identify the appropriate generalization of this duality, we first look at the fragmentation operation.
\subsection{Fragmentation}
The fragmentation operator $\mathrm{Frag}-\mathrm{PD}(\alpha,1-\alpha)$ can be defined generically as follows. For an input sequence  $(P_{i}),$ splitting of this sequence is achieved by attaching an independent $\mathrm{PD}(\alpha,1-\alpha)$ sequence, say $(Q_{i}),$ to the sized biased pick of $(P_{i}),$ say $\tilde{P}_{1}$ and then ranking the modified sequence. Hence for  the fixed input $(P_{i})=(p_{i}),$ let $(p^{*}_{k})$ denote the sequence remaining after the size biased pick $\tilde{p}_{1}$ is removed from $(p_{i}),$ then $\mathrm{Frag}-\mathrm{PD}(\alpha,1-\alpha)((p_{i}),\cdot)$ has the distribution equivalent to
\begin{equation}
\mathrm{Rank}(\tilde{p}_{1}(Q_{i}), (p^{*}_{k})).
\label{FragDGM}
\end{equation}
In terms of (\ref{DGMd}), the input follows a $\mathrm{PD}(\alpha,\theta)$ distribution and hence it is known that the distribution of the size biased pick $\tilde{P}_{1}\overset{d}=\beta_{1-\alpha,\theta+\alpha}.$ This can be expressed in terms of the following distributional equivalence that can be found in \cite{PPY92,Pit96,PY97}, see also \cite{JamesLamperti} for more details and references,
\begin{equation}
P_{\alpha,\theta}(y)\overset{d}=\beta_{\theta+\alpha,1-\alpha}P_{\alpha,\theta+\alpha}(y)+ (1-\beta_{\theta+\alpha,1-\alpha})\indic_{(U_{1}\leq y)},
\label{Stoch1}
\end{equation}
where the variables on the right hand side are independent. In this case, the distributional result for the fragmentation operation can be verified by the following result
\begin{equation}
P_{\alpha,1+\theta}(y)\overset{d}=\beta_{\theta+\alpha,1-\alpha}P_{\alpha,\theta+\alpha}(y)+ (1-\beta_{\theta+\alpha,1-\alpha})P_{\alpha,1-\alpha}(y),
\label{Stoch2}
\end{equation}
Similar to the case of the Pitman's $\mathrm{PD}(\alpha,-\alpha\delta)$ fragmentation operator that we applied to the  $\mathbb{P}_{\alpha}(\zeta)$ class in the previous section, we will show that the $\mathrm{Frag}-\mathrm{PD}(\alpha,1-\alpha)$ operator is natural to use in this present setting. In order to identify the appropriate distributional relations we will need to establish generalizations of equations ~(\ref{Stoch1}) and~(\ref{Stoch2}).
\begin{thm}Let $Q_{\alpha,\zeta}$ denote a $\mathbb{P}_{\alpha}(\zeta)$-bridge.
\begin{enumerate}
\item[(i)]Then,
\begin{equation}
Q_{\alpha,\zeta}(y)\overset{d}=\frac{\tau_{\alpha}(\zeta
y)}{\tau_{\alpha}(\zeta)}\overset{d}=\frac{\tau_{\alpha}((\gamma_{1}+\zeta)
y)+\gamma_{1-\alpha}\indic_{(U_{1}\leq
y)}}{\tau_{\alpha}(\gamma_{1}+\zeta)+\gamma_{1-\alpha}}
\label{idCoag}
\end{equation}
This can be written as,
$$
Q_{\alpha,\zeta}(y)\overset{d}=(1-\tilde{P}_{1})Q_{\alpha,\gamma_{1}+\zeta}(p)+\tilde{P}_{1}\indic_{(U_{1}\leq
y)}
$$
where the size biased pick from  fron a $\mathbb{P}_{\alpha}(\zeta)$ sequence can be represented as $\tilde{P}_{1}=\gamma_{1-\alpha}/(\tau_{\alpha}(\gamma_{1}+\zeta)+\gamma_{1-\alpha})$, and $Q_{\alpha,\gamma_{1}+\zeta}(y)=
\tau_{\alpha}((\gamma_{1}+\zeta)y)/\tau_{\alpha}(\gamma_{1}+\zeta)$ is a $\mathbb{P}_{\alpha}(\gamma_{1}+\zeta)$-bridge generally not independent of $\tilde{P}_{1}.$
\item[(ii)]Now replacing $\indic_{(U_{1}\leq
y)}$ by an independent $\mathrm{PD}(\alpha,1-\alpha)$-bridge it follows that a $\mathbb{P}_{\alpha}(\gamma_{1/\alpha}+\zeta)-$ bridge can be represented as
\begin{equation}
Q_{\alpha,\gamma_{1/\alpha}+\zeta}(y)\overset{d}=(1-\tilde{P}_{1})Q_{\alpha,\gamma_{1}+\zeta}(y)+\tilde{P}_{1}P_{\alpha,1-\alpha}(y).
\label{fragverify1}
\end{equation}
\item[(iii)]Furthermore, in terms of its marginal distribution, the size-biased pick $\tilde{P}_{1},$ which has distribution equivalent to the structural distribution of a $\mathbb{P}_{\alpha}(\zeta)$ sequence, can be represented as
\begin{equation}
\widetilde{P}_{1}\overset{d}=\beta_{1-\alpha,\alpha}
\frac{\gamma_{1}}{\gamma_{1}+\tau_{\alpha}(\zeta)}\overset{d}=\beta_{1-\alpha,\alpha}\left[1-\frac{\zeta^{1/\alpha}}{{(\zeta+\gamma_{1})}^{1/\alpha}}\right]
\label{structure}
\end{equation}
where the variables appearing on the right hand side are independent.
\end{enumerate}
\end{thm}
\begin{proof}Similar to \cite{Pit96} for (\ref{Stoch1}), statement [(i)] can be established by a Bayesian argument.
Let $X_{1}$ denote a variable so that conditional on $Q_{\alpha,\zeta}$ its distribution is $Q_{\alpha,\zeta}.$ Then noting that conditional on $\zeta,$ $Q_{\alpha,\zeta}$ is a generalized gamma bridge, it follows that the posterior distribution of $Q_{\alpha,\zeta}$ given $X_{1},\zeta$ can be read from James, Lijoi and Pr\"unster~\cite[Proposition 1, Theorems 1 and 2]{JLP2}. By scaling arguments, involving properties of the $\tau_{\alpha}$ subordinator, it follows that the unconditional distribution of $Q_{\alpha,\zeta}$ can be represented as,
$$
Q_{\alpha,\zeta}(y)=\frac{\tau_{\alpha}(\zeta
y)}{\tau_{\alpha}(\zeta)}\overset{d}=\frac{\tau_{\alpha}(\zeta{(1+\lambda)}^{\alpha}
y)+\gamma_{1-\alpha}\indic_{(U_{1}\leq
y)}}{\tau_{\alpha}(\zeta{(1+\lambda)}^{\alpha})+\gamma_{1-\alpha}}
$$
where $\lambda$ is a variable, appearing in James, Lijoi and Pr\"unster~\cite[Proposition 1]{JLP2} for $n=1,$ equal in distribution to $\gamma_{1}/\tau_{\alpha}(\zeta).$ From this, it is not difficult to see that the distribution of $(\lambda,\zeta)$ is given proportional to
$$
F_{\zeta}(dx)x{\mbox e}^{-x[(1+\lambda)^{\alpha}-1]}{(1+\lambda)}^{\alpha-1}
$$
Manipulating this distribution easily shows that
$$
\lambda\overset{d}=\zeta^{-1/\alpha}{(\gamma_{1}+\zeta)}^{1/\alpha}-1\overset{d}=\frac{\gamma_{1}}{\tau_{\alpha}(\zeta)}
$$
The identification of the size-biased pick is a consequence of the Bayesian argument.  For statement [(ii)], note that
since $\gamma_{1-\alpha}\overset{d}=\tau_{\alpha}(\gamma_{(1-\alpha)/\alpha})$ and is independent of $P_{\alpha,1-\alpha}(y),$ representable as $\tau_{\alpha}(\gamma_{(1-\alpha)/\alpha}y)/\tau_{\alpha}(\gamma_{(1-\alpha)/\alpha}),$ it follows that
$$
\tilde{P}_{1}P_{\alpha,1-\alpha}(y)=\frac{\tau_{\alpha}(\gamma_{(1-\alpha)/\alpha}y)}
{\tau_{\alpha}(\gamma_{1}+\zeta)+\tau_{\alpha}(\gamma_{(1-\alpha)/\alpha})}
$$
Pushing terms together and using the fact that $\gamma_{(1-\alpha)/\alpha}+\gamma_{1}\overset{d}=\gamma_{1/\alpha}$ completes the result. Statement [(iii)] follows from standard beta-gamma algebra and the results we discussed above.
\end{proof}
It is evident that the results in (\ref{idCoag}) and (\ref{fragverify1}) leads to the validity of the fragmentation operator.
\begin{thm}\label{frag2} Let $(P_{i}):=(\tilde{P}_{1},(P^{*}_{k})),$ where $\tilde{P}_{1}$ is obtained by sized biased sampling, denote a $\mathbb{P}_{\alpha}(\zeta)$ sequence and let $(Q_{i})$ denote an independent $\mathrm{PD}(\alpha,1-\alpha)$ sequence.
Then $\mathrm{Frag}-\mathrm{PD}(\alpha,1-\alpha)((P_{i}),\cdot),$ defined by~(\ref{FragDGM}), satisfies
$$
\mathrm{Rank}(\tilde{P}_{1}(Q_{i}), (P^{*}_{k}))\sim \mathbb{P}_{\alpha}(\gamma_{1/\alpha}+\zeta).
$$
\end{thm}
\begin{rem}The description of the structural distribution of $\mathbb{P}_{\alpha}(\zeta)$ given in (\ref{structure}) is new. When $\zeta$ is a constant the result gives an explicit description of the structural distribution of a generalized gamma process that is discussed in~\cite[p.15]{Pit02}.
\end{rem}
\subsection{Coagulation via simple bridges}
The fragmentation result now shows us that we need to find a coagulation operation such that when the random input is a $\mathbb{P}_{\alpha}(\gamma_{1/\alpha}+\zeta)$ sequence in $\mathcal{P}$, the resulting distribution of the operator is $\mathbb{P}_{\alpha}(\zeta).$ We first show that one can provide a generalization of the coagulation operator defined in
Dong, Goldschmidt, and
Martin(DGM)~\cite{Dong2006} using simple bridges.

We can describe this type of operator through the inverse of simple bridges. Recall the discussion on exchangeable bridges where a randomized simple bridge $b_{s_{1}}$ is defined in (\ref{simplebridge}). The inverse of a simple bridge is denoted as
$b^{-1}_{s_{1}}.$ From
Bertoin(\cite{BerFrag}, eq. (4.14), p. 194) one sees that for
$(U'_{k})_{k\ge 1}$ iid Uniform$[0,1]$ random variables
independent of $U_{1},$
$$b^{-1}_{s_{1}}(U'_{k})=U_{1}, {\mbox { iff }}
U'_{k}\in(s_{0}U_{1}, (1-s_{0})+s_{0}U_{1}),$$ having  length
$s_{1}=1-s_{0}$ and otherwise
$b^{-1}_{s_{1}}(U'_{k})\overset{d}=U^{0}_{k}$ has an independent
Uniform$[0,1]$ distribution, that is for $U'_{k}\in
[0,s_{0}U_{1}]\cup[s_{0}U_{1}+1-s_{0},1].$ More precisely, define
for each $k,$
\begin{equation}
I_{k}\overset{d}=\indic_{(b^{-1}_{s_{1}}(U'_{k})=U_{1})}\overset{d}=\indic_{(U'_{k}\leq
s_{1})} \label{indicator}
\end{equation}
then
$$
\mathbb{P}(b^{-1}_{s_{1}}(U'_{k})\leq y|I_{k}=0)=y, y\in [0,1].
$$
Now in general for some exchangeable bridge of the form $P(y)=\sum_{k=1}^{\infty}P_{k}\indic_{(U'_{k}\leq y)},$ one has
$$
P(b_{s_{1}}(y))\overset{d}=\sum_{k=1}^{\infty}P_{k}\indic_{(b^{-1}_{s_{1}}(U'_{k})\leq
y)}\overset{d}=\indic_{(U_{1}\leq
y)}[\sum_{k:I_{k}=1}P_{k}]+\sum_{\{k:I_{k}=0\}}P_{k}\indic_{(U^{0}_{k}\leq
y)}
$$
where
$$
P(s_{1})\overset{d}=\sum_{\{k:I_{k}=1\}}P_{k}.
$$
Hence the law of sequence $(Q_{k})\in \mathcal{P},$ such that $P\circ b_{s_{1}}(y)\overset{d}=\sum_{k=1}^{\infty}Q_{k}\indic_{(U_{k}\leq y)},$ can be expressed as
\begin{equation}
(Q_{i})\overset{d}=\mathrm{Rank}((P_{k}:
I_{k}=0), \sum_{\{k:I_{k}=1\}}P_{k}). \label{simplerank}
\end{equation} Hence by fixing $(P_{i})=(p_{i})$ in~(\ref{simplerank}), for an input $(p_{i})$ we define the Coag operator  $s_{1}-\mathrm{Coag}((p_{i}),\cdot)$ as the operator,
\begin{equation}
s_{1}-\mathrm{Coag}((p_{i}),\cdot):=\mathrm{Rank}((p_{k}:
I_{k}=0), \sum_{\{k:I_{k}=1\}}p_{k}). \label{EDGM} \end{equation}

When $s_{1}=\beta_{(1-\alpha)/\alpha, (\theta+\alpha)/\alpha},$~(\ref{EDGM}) coincides with the operator in Dong, Goldschmidt, and
Martin~\cite{Dong2006}. In that case, the beta variable is chosen apriori to be independent of the input. The next result shows that we need to choose $s_{1}$ generally dependent on the input.

\begin{thm}\label{coag2}Let $b_{s_{1}}$ denote a simple randomized bridge with
\begin{equation}
s_{1}:=\beta_{(\frac{1-\alpha}{\alpha},1)}\frac{\gamma_{1/\alpha}}{\gamma_{1/\alpha}+\zeta}
\label{bridges}
\end{equation}
where the beta variable is taken independent of the independent pair $(\gamma_{1/\alpha},\zeta),$ Using this same pair define the exchangeable bridge
$$
Q_{\alpha,\gamma_{1/\alpha}+\zeta}(y)=\frac{\tau_{\alpha}((\zeta+\gamma_{1/\alpha})
y)}{\tau_{\alpha}(\zeta+\gamma_{1/\alpha})}=\sum_{k=1}^{\infty}P_{k}\indic{(U_{k}\leq p)}{\mbox { and }}T=\frac{\tau_{\alpha}(\gamma_{1/\alpha}+\zeta)}{{(\gamma_{1/\alpha}+\zeta)}^{1/\alpha}}
$$
Marginally $Q_{\alpha,\gamma_{1/\alpha}+\zeta}$ is a $\mathbb{P}_{\alpha}(\gamma_{1/\alpha}+\zeta)$-bridge with $\alpha$-DIVERSITY $T^{-\alpha}.$
\begin{enumerate}
\item[(i)]Then, for each $y\in [0,1]$
$$
Q_{\alpha,\gamma_{1/\alpha}+\zeta}(b_{s_{1}}(y))\overset{d}=Q_{\alpha,\zeta}(y),
$$
where $Q_{\alpha,\zeta}$ is a $\mathbb{P}_{\alpha}(\zeta)$-bridge.
\item[(ii)]Then it follows that $s_{1}-\mathrm{Coag}((P_{k}),\cdot)$ has distribution,
$$
\mathrm{Rank}((P_{k}:
I_{k}=0), \sum_{\{k:I_{k}=1\}}P_{k})\sim\mathbb{P}_{\alpha}(\zeta)$$
\item[(iii)]Let $(P_{k})=(p^{(v)}_{k})$ denote the realization such that $T=v$ then the $s_{1}-\mathrm{Coag}((p^{(v)}_{k}),\cdot)$
 given $(p^{(v)}_{k})$ is equivalent in distribution
 $$
 \mathrm{Rank}((p^{(v)}_{k}:
I^{(v)}_{k}=0), \sum_{\{k:I^{(v)}_{k}=1\}}p^{(v)}_{k}),$$
where $I^{(v)}_{k}\overset{d}=\indic_{(U_{k}\leq s^{(v)}_{1})},$ and $s^{(v)}_{1}$ has the conditional distribution of $s_{1}$ given $(p^{(v)}_{k}).$ In particular the distribution of $s^{(v)}_{1}$ equates with the distribution of $s_{1}|T=v.$ So
$$
s^{(v)}_{1}\overset{d}=\beta_{(\frac{1-\alpha}{\alpha},1)}W^{(v)}
$$
where, for $y\in [0,1],$ the density of $1-W^{(v)}$ is given proportional to
$$
{(1-y)}^{1/\alpha-1}y^{-1/\alpha-1}
\mathbb{E}[{\mbox e}^{-vy^{1/\alpha}\zeta^{1/\alpha}}{\mbox e}^{\zeta}\zeta^{1/\alpha}].
$$
\end{enumerate}
\end{thm}
\begin{proof}Statement [(i)] follows from the equivalence in~(\ref{idCoag}), since it is easy to see that, for $y\in [0,1],$
$$
Q_{\alpha,\gamma_{1/\alpha}+\zeta}(b_{s_{1}}(y))=\frac{\tau_{\alpha}((\gamma_{1}+\zeta)
y)+\tau_{\alpha}(\gamma_{(1-\alpha)/\alpha}\indic_{(U_{1}\leq
y)})}{\tau_{\alpha}(\gamma_{1}+\zeta)+\tau_{\alpha}(\gamma_{(1-\alpha)/\alpha})}
$$
and $\tau_{\alpha}(\gamma_{(1-\alpha)/\alpha}\indic_{(U_{1}\leq
y)})\overset{d}=\gamma_{1-\alpha}\indic_{(U_{1}\leq
y)}.$ [(ii)] is immediate from [(i)]. For [(iii)], we again appeal to Pitman and Yor~\cite[p. 877, see eq. (96) and (97)]{PY97}. That is, conditioning on $T$ it follows that $(P_{k})$ and $s_{1}$ are conditionally independent. It is then straightforward to obtain the conditional density of $s_{1}$ given $T.$
\end{proof}
\subsection{Duality}We can now describe the duality relation in terms of the following diagram
\begin{eqnarray}
  & \beta_{(\frac{1-\alpha}{\alpha},1)}\frac{\gamma_{1/\alpha}}{\gamma_{1/\alpha}+\zeta}-\mathrm{Coag} &  \nonumber\\
  \mathbb{P}_{\alpha}(\gamma_{1/\alpha}+\zeta) & \DoubleRLarrow{0.5in} & \mathbb{P}_{\alpha}(\zeta) \label{DGMgen}\\
  & \mathrm{Frag}-\mathrm{PD}(\alpha,1-\alpha) &\nonumber
\end{eqnarray}
For $\theta\ge 0$ this reduces to (\ref{DGMd}) by setting $\zeta\overset{d}=\gamma_{\theta/\alpha}.$ Furthermore setting $\zeta=\gamma_{(n-1)/\alpha}+\zeta,$ in~(\ref{DGMgen}) leads to a recursion representable as,
\begin{eqnarray}
  & \beta_{(\frac{1-\alpha}{\alpha},\frac{n-1+\alpha}{\alpha})}\frac{\gamma_{n/\alpha}}{\gamma_{n/\alpha}+\zeta}-\mathrm{Coag} &  \nonumber\\
  \mathbb{P}_{\alpha}(\gamma_{n/\alpha}+\zeta) & \DoubleRLarrow{0.5in} & \mathbb{P}_{\alpha}(\gamma_{(n-1)/\alpha}+\zeta) \nonumber\\
  & \mathrm{Frag}-\mathrm{PD}(\alpha,1-\alpha) &\nonumber
\end{eqnarray}
where
$$
\beta_{(\frac{1-\alpha}{\alpha},\frac{n-1+\alpha}{\alpha})}\frac{\gamma_{n/\alpha}}{\gamma_{n/\alpha}+\zeta}\overset{d}=\beta_{(\frac{1-\alpha}{\alpha},1)}\frac{\gamma_{1/\alpha}}{\gamma_{n/\alpha}+\zeta}.$$
We close with a formal statement.

\begin{thm}\label{Dual}Suppose that $X$ and $Y$ are sequences in ${\mathcal P}.$
Then, using the descriptions in Theorems~\ref{frag2} and~\ref{coag2}, the following statements are equivalent.
\begin{enumerate}
\item[(i)]$X\sim \mathbb{P}_{\alpha}(\zeta)$ and
conditional on $X,$
$Y\overset{d}=\mathrm{Frag}-\mathrm{PD}(\alpha,1-\alpha)(X,\cdot).$
\item[(ii)]$Y\sim \mathbb{P}_{\alpha}(\gamma_{1/\alpha}+\zeta)$ and conditional on $Y,$
$X\sim
\beta_{(\frac{1-\alpha}{\alpha},1)}\frac{\gamma_{1/\alpha}}{\gamma_{1/\alpha}+\zeta}-\mathrm{Coag}(Y,\cdot).$

Where in particular, for $Y=(p^{(v)}_{k}),$ indicating its $\alpha-$DIVERSITY or local time has value $v^{-\alpha},$
$$
\beta_{(\frac{1-\alpha}{\alpha},1)}\frac{\gamma_{1/\alpha}}{\gamma_{1/\alpha}+\zeta}-(\mathrm{Coag}((p^{(v)}_{k}),\cdot)
\overset{d}=s^{(v)}_{1}-(\mathrm{Coag}((p^{(v)}_{k}),\cdot).
$$
Which is described in [(iii)] of Theorem~\ref{coag2}.
\end{enumerate}
\end{thm}


\end{document}